\documentclass[10pt,a4paper]{article}
\usepackage[latin1]{inputenc}
\usepackage{amsmath}
\usepackage{amsfonts}
\usepackage{amssymb,amsbsy,amsthm}

\usepackage{enumerate}

\author{Mohamed Khaled and Tarek Sayed Ahmed\\ Department of Mathematics, Faculty of Science,\\
Cairo University, Giza, Egypt.}
\title{Strongly representable atom structures}
\date{}
\newtheorem{thm}{Theorem}[section]
\newtheorem{lem}{Lemma}[section]

\numberwithin{equation}{section}
\newtheorem{defn}{Definition}[section]

\def\A{{\mathfrak{A}}}
\begin{document}

\maketitle
\begin{abstract}
An atom structure of type $\mathcal{T}$ is said to be strongly representable if all atomic algebras (of the same type $\mathcal{T}$) with that atom structure are representable.
We show that for any finite $n\geq 3$ and any signature $\mathcal{T}$ between $Df_{n}$
and $QEA_{n}$, the class of strongly representable atom structures of type $\mathcal{T}$ is not elementary. We extensively use graphs and games as introduced in algebraic logic by Hirsch and Hodkinson.
\end{abstract}
\section{Introduction}
In \cite{hirsh}, Hirsch and Hodkinson proved that for finite $n\geq 3$, the class of strongly
representable cylindric-type atom structures of dimension $n$ is not definable by
any set of first-order sentences: it is not elementary class. Their method depends
on that $RCA_n$ is a variety, an atomic algebra $\A$ will be in $RCA_n$ if all the
equations defining $RCA_n$ are valid in $\A$. From the point of view of $At\A$, each
equation corresponds to a certain universal monadic second-order statement,
where the universal quantifiers are restricted to ranging over the sets of atoms
that are defined by elements of $\A$. Such a statement will fail in $\A$ if $At\A$ can be
partitioned into finitely many $\A$-definable sets with certain properties - they call
this a bad partition. This idea can be used to show that $RCA_n$ (for $n\geq 3$) is
not finitely axiomatizable, by finding a sequence of atom structures, each having
some sets that form a bad partition, but with the minimal number of sets in a
bad partition increasing as we go along the sequence. This can yield algebras
not in $RCA_n$ but with an ultraproduct that is in $RCA_n$. In this article we
extend the result of Hirsch and Hodkinson to any class of strongly representable
atom structure having signature between the diagonal free atom structures and
the quasi polyadic equality atom structures (recall the definitions of such algebras from \cite{tarski} and \cite{thompson}). As in \cite{hirsh} we deal only with finite
dimensional algebras. Fix a finite dimension $n <\omega$, with $n\geq 3$.
\section{Atom structures}
The action of the non-boolean operators in a completely additive
atomic $BAO$ is determined by their behavior over the atoms, and
this in turn is encoded by the atom structure of the algebra.
\begin{defn}(\textbf{Atom Structure})\\
Let $\mathcal{A}=\langle A, +, -, 0, 1, \Omega_{i}:i\in I\rangle$ be
an atomic boolean algebra with operators $\Omega_{i}:i\in I$. Let
the rank of $\Omega_{i}$ be $\rho_{i}$. The \textit{atom structure}
$At\mathcal{A}$ of $\mathcal{A}$ is a relational structure
\begin{center}$\langle At\mathcal{A}, R_{\Omega_{i}}:i\in I\rangle$\end{center}
where $At\mathcal{A}$ is the set of atoms of $\mathcal{A}$ as
before, and $R_{\Omega_{i}}$ is a $(\rho(i)+1)$-ary relation over
$At\mathcal{A}$ defined by\begin{equation*}R_{\Omega_{i}}(a_{0},
\cdots, a_{\rho(i)})\Longleftrightarrow\Omega_{i}(a_{1}, \cdots,
a_{\rho(i)})\geq a_{0}.\end{equation*}
\end{defn}
Similar 'dual' structure arise in other ways, too. For any not
necessarily atomic $BAO$ $\mathcal{A}$ as above, its
\textit{ultrafilter frame} is the
structure\begin{equation*}\mathcal{A}_{+}=\langle Uf(\mathcal{A}),
R_{\Omega_{i}}:i\in I\rangle,\end{equation*} where $Uf(\mathcal{A})$
is the set of all ultrafilters of (the boolean reduct of)
$\mathcal{A}$, and for $\mu_{0}, \cdots, \mu_{\rho(i)}\in
Uf(\mathcal{A})$, we put $R_{\Omega_{i}}(\mu_{0}, \cdots,
\mu_{\rho(i)})$ iff $\{\Omega(a_{1}, \cdots,
a_{\rho(i)}):a_{j}\in\mu_{j}$ for
$0<j\leq\rho(i)\}\subseteq\mu_{0}$.
\begin{defn}(\textbf{Complex algebra})\\
Conversely, if we are given an arbitrary structure
$\mathcal{S}=\langle S, r_{i}:i\in I\rangle$ where $r_{i}$ is a
$(\rho(i)+1)$-ary relation over $S$, we can define its
\textit{complex
algebra}\begin{equation*}\mathfrak{Cm}(\mathcal{S})=\langle \wp(S),
\cup, \setminus, \phi, S, \Omega_{i}\rangle_{i\in
I},\end{equation*}where $\wp(S)$ is the power set of $S$, and
$\Omega_{i}$ is the $\rho(i)$-ary operator defined
by\begin{equation*}\Omega_{i}(X_{1}, \cdots, X_{\rho(i)})=\{s\in
S:\exists s_{1}\in X_{1}\cdots\exists s_{\rho(i)}\in X_{\rho(i)},
r_{i}(s, s_{1}, \cdots, s_{\rho(i)})\},\end{equation*} for each
$X_{1}, \cdots, X_{\rho(i)}\in\wp(S)$.
\end{defn}
It is easy to check that, up to isomorphism,
$At(\mathfrak{Cm}(\mathcal{S}))\cong\mathcal{S}$ always, and
$\mathcal{A}\subseteq\mathfrak{Cm}(At\mathcal{A})$ for any
completely additive atomic $BAO$ $\mathcal{A}$. If $\mathcal{A}$ is
finite then of course
$\mathcal{A}\cong\mathfrak{Cm}(At\mathcal{A})$.\\
\begin{itemize}
\item{Atom structure of diagonal free-type algebra is $\mathcal{S}=\langle S, R_{c_{i}}:i<n\rangle$, where the $R_{c_{i}}$
is binary relation on $S$.}
\item{Atom structure of cylindric-type algebra is $\mathcal{S}=\langle S, R_{c_{i}}, R_{d_{ij}}:i, j<n\rangle$, where the
$R_{d_{ij}}$, $R_{c_{i}}$ are unary and binary relations on $S$. The
reduct $\mathfrak{Rd}_{df}\mathcal{S}=\langle S,
R_{c_{i}}:i<n\rangle$ is an atom structure of diagonal free-type.}
\item{Atom structure of substitution-type algebra is $\mathcal{S}=\langle S, R_{c_{i}}, R_{s^{i}_{j}}:i, j<n\rangle$, where the
$R_{d_{ij}}$, $R_{s^{i}_{j}}$ are unary and binary relations on $S$,
respectively. The reduct $\mathfrak{Rd}_{df}\mathcal{S}=\langle S,
R_{c_{i}}:i<n\rangle$ is an atom structure of diagonal free-type.}
\item{Atom structure of quasi polyadic-type algebra is $\mathcal{S}=\langle S, R_{c_{i}}, R_{s^{i}_{j}}$, $R_{s_{ij}}:i, j<n\rangle$, where the
$R_{c_{i}}$, $R_{s^{i}_{j}}$ and $R_{s_{ij}}$ are binary relations
on $S$. The reducts $\mathfrak{Rd}_{df}\mathcal{S}=\langle S,
R_{c_{i}}:i<n\rangle$ and $\mathfrak{Rd}_{Sc}\mathcal{S}=\langle S,
R_{c_{i}}, R_{s^{i}_{j}}:i, j<n\rangle$ are atom structures of
diagonal free and substitution types, respectively.}
\item{Atom structure of quasi polyadic equality-type algebra is $\mathcal{S}=\langle S, R_{c_{i}}, R_{d_{ij}}, R_{s^{i}_{j}}, R_{s_{ij}}:i, j<n\rangle$, where the $R_{d_{ij}}$
is unary relation on $S$, and $R_{c_{i}}$, $R_{s^{i}_{j}}$ and
$R_{s_{ij}}$ are binary relations on $S$.\begin{itemize}\item{The
reduct $\mathfrak{Rd}_{df}\mathcal{S}=\langle S, R_{c_{i}}:i\in
I\rangle$ is an atom structure of diagonal free-type.}
\item{The
reduct $\mathfrak{Rd}_{ca}\mathcal{S}=\langle S, R_{c_{i}},
R_{d_{ij}}:i, j\in I\rangle$ is an atom structure of
cylindric-type.}
\item{The
reduct $\mathfrak{Rd}_{Sc}\mathcal{S}=\langle S, R_{c_{i}},
R_{s^{i}_{j}}:i, j\in I\rangle$ is an atom structure of
substitution-type.}
\item{The
reduct $\mathfrak{Rd}_{qa}\mathcal{S}=\langle S, R_{c_{i}},
R_{s^{i}_{j}}, R_{s_{ij}}:i, j\in I\rangle$ is an atom structure of
quasi polyadic-type.}
\end{itemize}}
\end{itemize}
\begin{defn}
An algebra is said to be representable if and only if it is isomorphic to a subalgebra of a direct product of set algebras of the same type.
\end{defn}
\begin{defn}\label{strong rep}Let $\mathcal{S}$ be an $n$-dimensional algebra atom
structure. $\mathcal{S}$ is \textit{strongly representable} if every
atomic $n$-dimensional algebra $\mathcal{A}$ with
$At\mathcal{A}=\mathcal{S}$ is representable. We write $SDfS_{n}$,
$SCS_{n}$, $SSCS_{n}$, $SQS_{n}$ and $SQES_{n}$ for the classes of
strongly representable ($n$-dimensional) diagonal free, cylindric,
substitution, quasi polyadic and quasi polyadic equality algebra
atom structures, respectively.
\end{defn}
Note that for any $n$-dimensional algebra $\mathcal{A}$ and atom
structure $\mathcal{S}$, if $At\mathcal{A}=\mathcal{S}$ then
$\mathcal{A}$ embeds into $\mathfrak{Cm}\mathcal{S}$, and hence
$\mathcal{S}$ is strongly representable iff
$\mathfrak{Cm}\mathcal{S}$ is representable.
\section{Graphs and Strong representability}
In this section, by a graph we will mean a pair $\Gamma=(G, E)$,
where $G\not=\phi$ and $E\subseteq G\times G$ is a reflexive and
symmetric binary relation on $G$. We will often use the same
notation for $\Gamma$ and for its set of nodes ($G$ above). A pair
$(x, y)\in E$ will be called an edge of $\Gamma$. See \cite{graph}
for basic information (and a lot more) about graphs.
\begin{defn}
Let $\Gamma=(G, E)$ be a graph.
\begin{enumerate}
\item{A set $X\subset G$ is said to be \textit{independent} if $E\cap(X\times X)=\phi$.}
\item{The \textit{chromatic number} $\chi(\Gamma)$ of $\Gamma$ is the smallest $\kappa<\omega$ such that $G$ can be partitioned into $\kappa$ independent sets,
and $\infty$ if there is no such $\kappa$.}
\end{enumerate}
\end{defn}

\begin{defn}\ \begin{itemize}
\item For an equivalence relation $\sim$ on a set $X$, and $Y\subseteq
X$, we write $\sim\upharpoonright Y$ for $\sim\cap(Y\times Y)$. For
a partial map $K:n\rightarrow\Gamma\times n$ and $i, j<n$, we write
$K(i)=K(j)$ to mean that either $K(i)$, $K(j)$ are both undefined,
or they are both defined and are equal.
\item For any two relations $\sim$ and $\approx$. The composition of $\sim$ and $\approx$ is
the set
\begin{equation*}\sim\circ\approx=\{(a, b):\exists c(a\sim c\wedge c\approx b)\}.\end{equation*}\end{itemize}
\end{defn}
\begin{defn}Let $\Gamma$ be a graph.
We define an atom structure $\eta(\Gamma)=\langle H, D_{ij},
\equiv_{i}, \equiv_{ij}:i, j<n\rangle$ as follows:
\begin{enumerate}
\item{$H$ is the set of all pairs $(K, \sim)$ where $K:n\rightarrow \Gamma\times n$ is a partial map and $\sim$ is an equivalent relation on $n$
satisfying the following conditions \begin{enumerate}\item{If
$|n\diagup\sim|=n$, then $dom(K)=n$ and $rng(K)$ is not independent
subset of $n$.}
\item{If $|n\diagup\sim|=n-1$, then $K$ is defined only on the unique $\sim$ class $\{i, j\}$ say of size $2$ and $K(i)=K(j)$.}
\item{If $|n\diagup\sim|\leq n-2$, then $K$ is nowhere defined.}
\end{enumerate}}
\item{$D_{ij}=\{(K, \sim)\in H : i\sim j\}$.}
\item{$(K, \sim)\equiv_{i}(K', \sim')$ iff $K(i)=K'(i)$ and $\sim\upharpoonright(n\setminus\{i\})=\sim'\upharpoonright(n\setminus\{i\})$.}
\item{$(K, \sim)\equiv_{ij}(K', \sim')$ iff $K(i)=K'(j)$, $K(j)=K'(i)$, and $K(\kappa)=K'(\kappa) (\forall\kappa\in n\setminus\{i, j\})$
and if $i\sim j$ then $\sim=\sim'$, if not, then $\sim'=\sim\circ[i,
j]$.}
\end{enumerate}
\end{defn}
It may help to think of $K(i)$ as assigning the nodes $K(i)$ of
$\Gamma\times n$ not to $i$ but to
the set $n\setminus\{i\}$, so long as its elements are pairwise non-equivalent via $\sim$.\\
For a set $X$, $\mathcal{B}(X)$ denotes the boolean algebra
$\langle\wp(X), \cup, \setminus\rangle$. We write $a\cap b$ for
$-(-a\cup-b)$.
\begin{defn}\label{our algebra}
Let $\mathfrak{B}(\Gamma)=\langle\mathcal{B}(\eta(\Gamma)), c_{i},
s^{i}_{j}, s_{ij}, d_{ij}\rangle_{i, j<n}$ be the algebra, with
extra non-Boolean operations defined as follows:
\begin{center}
$d_{ij}=D_{ij}$,\\
$c_{i}X=\{c: \exists a\in X, a\equiv_{i}c\}$,\\
$s_{ij}X=\{c: \exists a\in X, a\equiv_{ij}c\}$,\\
$s^{i}_{j}X=\begin{cases}
c_{i}(X\cap D_{ij}), &\text{if $i\not=j$,}\\
X, &\text{if $i=j$.}
\end{cases}$
For all $X\subseteq \eta(\Gamma)$.
\end{center}
\end{defn}

\begin{defn}
For any $\tau\in\{\pi\in n^{n}: \pi \text{ is a bijection}\}$, and
any $(K, \sim)\in\eta(\Gamma)$. We define $\tau(K,
\sim)=(K\circ\tau, \sim\circ\tau)$.\end{defn}
The proof of the following two Lemmas is straightforward.
\begin{lem}\label{Lemma 1}\ \\
For any $\tau\in\{\pi\in n^{n}: \pi \text{ is a bijection}\}$, and
any $(K, \sim)\in\eta(\Gamma)$. $\tau(K, \sim)\in\eta(\Gamma)$.
\end{lem}
\begin{lem}\label{Lemma 2}\ \\For any $(K, \sim)$, $(K', \sim')$, and $(K'', \sim'')\in\eta(\Gamma)$, and $i, j\in n$:
\begin{enumerate}
\item{$(K, \sim)\equiv_{ii}(K', \sim')\Longleftrightarrow (K, \sim)=(K', \sim')$.}
\item{$(K, \sim)\equiv_{ij}(K', \sim')\Longleftrightarrow (K, \sim)\equiv_{ji}(K', \sim')$.}
\item{If $(K, \sim)\equiv_{ij}(K', \sim')$, and $(K, \sim)\equiv_{ij}(K'', \sim'')$, then $(K', \sim')=(K'', \sim'')$.}
\item{If $(K, \sim)\in D_{ij}$, then \\
$(K, \sim)\equiv_{i}(K', \sim')\Longleftrightarrow\exists(K_{1},
\sim_{1})\in\eta(\Gamma):(K, \sim) \equiv_{j}(K_{1},
\sim_{1})\wedge(K', \sim')\equiv_{ij}(K_{1}, \sim_{1})$.}
\item{$s_{ij}(\eta(\Gamma))=\eta(\Gamma)$.}
\end{enumerate}
\end{lem}
\begin{thm}\label{it is qea}
For any graph $\Gamma$, $\mathfrak{B}(\Gamma)$ is a simple
$QEA_{n}$.
\end{thm}
\begin{proof}\
We follow the axiomatization in \cite{thompson} except renaming the items by $Q_i$.
Let $X\subseteq\eta(\Gamma)$, and $i, j, \kappa\in n$:
\begin{itemize}
\item{$s^{i}_{i}=ID$ by definition \ref{our algebra}, $s_{ii}X=\{c:\exists a\in X, a\equiv_{ii}c\}=\{c:\exists a\in X, a=c\}=X$
(by Lemma \ref{Lemma 2} (1));\\
$s_{ij}X=\{c:\exists a\in X, a\equiv_{ij}c\}=\{c:\exists a\in X,
a\equiv_{ji}c\}=s_{ji}X$ (by Lemma \ref{Lemma 2} (2)).}
\item{Axioms $Q_{1}$, $Q_{2}$ follow directly from the fact that the reduct $\mathfrak{Rd}_{ca}\mathfrak{B}(\Gamma)=\langle\mathcal{B}(\eta(\Gamma)), c_{i}$, $d_{ij}\rangle_{i, j<n}$
is a cylindric algebra which is proved in \cite{hirsh}.}
\item{Axioms $Q_{3}$, $Q_{4}$, $Q_{5}$ follow from the fact that the reduct $\mathfrak{Rd}_{ca}\mathfrak{B}(\Gamma)$ is a cylindric algebra, and from \cite{tarski}
(Theorem 1.5.8(i), Theorem 1.5.9(ii), Theorem 1.5.8(ii)).}
\item{$s^{i}_{j}$ is a boolean endomorphism by \cite{tarski} (Theorem 1.5.3).
\begin{eqnarray*}
s_{ij}(X\cup Y)=&&\{c:\exists a\in(X\cup Y), a\equiv_{ij}c\}\\
=&&\{c:(\exists a\in X\vee\exists a\in Y), a\equiv_{ij}c\}\\
=&&\{c:\exists a\in X, a\equiv_{ij}c\}\cup\{c:\exists a\in Y,
a\equiv_{ij}c\}\\
=&&s_{ij}X\cup s_{ij}Y.
\end{eqnarray*}
$s_{ij}(-X)=\{c:\exists a\in(-X), a\equiv_{ij}c\}$, and
$s_{ij}X=\{c:\exists a\in X, a\equiv_{ij}c\}$ are disjoint. For, let
$c\in(s_{ij}(X)\cap s_{ij}(-X))$, then $\exists a\in X\wedge b\in
(-X)$, such that $a\equiv_{ij}c$, and $b\equiv_{ij}c$. Then $a=b$, (by
Lemma \ref{Lemma 2} (3)), which is a contradiction. Also,
\begin{eqnarray*}
s_{ij}X\cup s_{ij}(-X)=&&\{c:\exists a\in X,
a\equiv_{ij}c\}\cup\{c:\exists a\in(-X), a\equiv_{ij}c\}\\
=&&\{c:\exists a\in(X\cup-X), a\equiv_{ij}c\}\\
=&&s_{ij}\eta(\Gamma)\\
=&&\eta(\Gamma). \text{ (by Lemma \ref{Lemma 2} (5))}
\end{eqnarray*}
therefore, $s_{ij}$ is a boolean endomorphism.}
\item{\begin{eqnarray*}s_{ij}s_{ij}X=&&s_{ij}\{c:\exists a\in X, a\equiv_{ij}c\}\\
=&&\{b:(\exists a\in X\wedge c\in\eta(\Gamma)), a\equiv_{ij}c, \text{ and }
c\equiv_{ij}b\}\\
=&&\{b:\exists a\in X, a=b\}\\
=&&X.\end{eqnarray*}}
\item{\begin{eqnarray*}s_{ij}s^{i}_{j}X=&&\{c:\exists a\in s^{i}_{j}X, a\equiv_{ij}c\}\\
=&&\{c:\exists b\in(X\cap d_{ij}),a\equiv_{i}b\wedge
a\equiv_{ij}c\}\\
=&&\{c:\exists b\in(X\cap d_{ij}), c\equiv_{j}b\} \text{ (by Lemma
\ref{Lemma 2} (4))}\\
=&&s^{j}_{i}X.\end{eqnarray*}}
\item{We need to prove that $s_{ij}s_{i\kappa}X=s_{j\kappa}s_{ij}X$ if $|\{i, j, \kappa\}|=3$. For, let $(K, \sim)\in s_{ij}s_{i\kappa}X$ then
$\exists(K', \sim')\in\eta(\Gamma)$, and $\exists(K'', \sim'')\in X$
such that $(K'', \sim'')\equiv_{i\kappa}(K', \sim')$ and $(K',
\sim')\equiv_{ij}(K, \sim)$.\\
Define $\tau:n\rightarrow n$ as follows:
\begin{eqnarray*}\tau(i)=&&j\\
\tau(j)=&&\kappa\\
\tau(\kappa)=&&i, \text{ and}\\
\tau(l)=&&l \text{ for every } l\in(n\setminus\{i, j,
\kappa\}).\end{eqnarray*} Now, it is easy to verify that $\tau(K',
\sim')\equiv_{ij}(K'', \sim'')$, and $\tau(K',
\sim')\equiv_{j\kappa}(K, \sim)$. Therefore, $(K, \sim)\in
s_{j\kappa}s_{ij}X$, i.e., $s_{ij}s_{i\kappa}X\subseteq
s_{j\kappa}s_{ij}X$. Similarly, we can show that
$s_{j\kappa}s_{ij}X\subseteq s_{ij}s_{i\kappa}X$.}
\item{Axiom $Q_{10}$ follows from \cite{tarski} (Theorem 1.5.7)}
\item{Axiom $Q_{11}$ follows from axiom 2, and the definition of $s^{i}_{j}$.}
\end{itemize}
Since $\mathfrak{Rd}_{ca}\mathfrak{B}$ is a simple $CA_{n}$, by
\cite{hirsh}, then $\mathfrak{B}$ is simple.
\end{proof}
\begin{defn}
Let $\mathfrak{C}(\Gamma)$ be the subalgebra of
$\mathfrak{B}(\Gamma)$ generated by the set of atoms.
\end{defn}
Note that the cylindric algebra constructed in \cite{hirsh} is
$\mathfrak{Rd}_{ca}\mathfrak{B}(\Gamma)$ not
$\mathfrak{Rd}_{ca}\mathfrak{C}(\Gamma)$, but all results in
\cite{hirsh} can be applied to
$\mathfrak{Rd}_{ca}\mathfrak{C}(\Gamma)$. Therefore, since our
results depends basically on \cite{hirsh}, we will refer to
\cite{hirsh} directly when we apply it to catch any result about
$\mathfrak{Rd}_{ca}\mathfrak{C}(\Gamma)$.
\begin{thm}
$\mathfrak{C}(\Gamma)$ is a simple $QEA_{n}$ generated by the set of the
$n-1$ dimensional elements.
\end{thm}
\begin{proof}
$\mathfrak{C}(\Gamma)$ is a simple $QEA_{n}$ from  Theorem \ref{it
is qea}. It remains to show that $\{(K, \sim)\}=\prod\{c_{i}\{(K,
\sim)\}: i<n\}$ for any $(K, \sim)\in H$. Let $(K, \sim)\in H$,
clearly $\{(K, \sim)\}\leq\prod\{c_{i}\{(K, \sim)\}: i<n\}$. For the
other direction assume that $(K', \sim')\in H$ and $(K,
\sim)\not=(K', \sim')$. We show that $(K',
\sim')\not\in\prod\{c_{i}\{(K, \sim)\}: i<n\}$. Assume toward a
contradiction that $(K', \sim')\in\prod\{c_{i}\{(K, \sim)\}:i<n\}$,
then $(K', \sim')\in c_{i}\{(K, \sim)\}$ for all $i<n$, i.e.,
$K'(i)=K(i)$ and
$\sim'\upharpoonright(n\setminus\{i\})=\sim\upharpoonright(n\setminus\{i\})$
for all $i<n$. Therefore, $(K, \sim)=(K', \sim')$ which makes a
contradiction, and hence we get the other direction.
\end{proof}
\begin{thm}\label{chr. no.}
Let $\Gamma$ be a graph.
\begin{enumerate}\item{Suppose that $\chi(\Gamma)=\infty$. Then $\mathfrak{C}(\Gamma)$
is representable.}\item{If $\Gamma$ is infinite and
$\chi(\Gamma)<\infty$ then $\mathfrak{Rd}_{df}\mathfrak{C}$ is not
representable.}\end{enumerate}
\end{thm}
\begin{proof}\
\begin{enumerate}
\item{We have $\mathfrak{Rd}_{ca}\mathfrak{C}$ is representable (c.f., \cite{hirsh}). Let $X=\{x\in \mathfrak{C}:\Delta x\not=n\}$. Call $J\subseteq \mathfrak{C}$ inductive if
$X\subseteq J$ and $J$ is closed under infinite unions and
complementation. Then $\mathfrak{C}$ is the smallest inductive
subset of $\mathfrak{C}$. Let $f$ be an isomorphism of
$\mathfrak{Rd}_{ca}\mathfrak{C}$ onto a cylindric set algebra with
base $U$. Clearly, by definition, $f$ preserves $s^{i}_{j}$ for each
$i, j<n$. It remains to show that $f$ preserves $s_{ij}$ for every
$i, j<n$. Let $i, j<n$, since $s_{ij}$ is boolean endomorphism and
completely additive, it suffices to show that $fs_{ij}x=s_{ij}fx$
for all $x\in At\mathfrak{C}$. Let $x\in At\mathfrak{C}$ and $\mu\in
n\setminus\Delta x$. If $\kappa=\mu$ or $l=\mu$, say $\kappa=\mu$,
then\begin{equation*} fs_{\kappa l}x=fs_{\kappa
l}c_{\kappa}x=fs^{\kappa}_{l}x=s^{\kappa}_{l}fx=s_{\kappa l}fx.
\end{equation*}
If $\mu\not\in\{\kappa, l\}$ then\begin{equation*} fs_{\kappa
l}x=fs^{l}_{\mu}s^{\kappa}_{l}s^{\mu}_{\kappa}c_{\mu}x=s^{l}_{\mu}s^{\kappa}_{l}s^{\mu}_{\kappa}c_{\mu}fx=s_{\kappa
l}fx.
\end{equation*}}
\item{Assume toward a contradiction that $\mathfrak{Rd}_{df}\mathfrak{C}$ is representable. Since $\mathfrak{Rd}_{ca}\mathfrak{C}$
is generated by $n-1$ dimensional elements then
$\mathfrak{Rd}_{ca}\mathfrak{C}$ is representable. But this
contradicts Proposition 5.4 in \cite{hirsh}.}
\end{enumerate}
\end{proof}
\begin{thm}\ \\
Let $2<n<\omega$ and $\mathcal{T}$ be any signature between $Df_{n}$
and $QEA_{n}$. Then the class of strongly representable atom
structures of type $\mathcal{T}$ is not elementary.
\end{thm}
\begin{proof}
By Erd\"{o}s's famous 1959 Theorem \cite{Erdos}, for each finite
$\kappa$ there is a finite graph $G_{\kappa}$ with
$\chi(G_{\kappa})>\kappa$ and with no cycles of length $<\kappa$.
Let $\Gamma_{\kappa}$ be the disjoint union of the $G_{l}$ for
$l>\kappa$. Clearly, $\chi(\Gamma_{\kappa})=\infty$. So by Theorem
\ref{chr.
no.} (1), $\mathfrak{C}(\Gamma_{\kappa})=\mathfrak{C}(\Gamma_{\kappa})^{+}$ is representable.\\
\- Now let $\Gamma$ be a non-principal ultraproduct
$\prod_{D}\Gamma_{\kappa}$ for the $\Gamma_{\kappa}$. It is
certainly infinite. For $\kappa<\omega$, let $\sigma_{\kappa}$ be a
first-order sentence of the signature of the graphs. stating that
there are no cycles of length less than $\kappa$. Then
$\Gamma_{l}\models\sigma_{\kappa}$ for all $l\geq\kappa$. By
{\L}o\'{s}'s Theorem, $\Gamma\models\sigma_{\kappa}$ for all
$\kappa$. So $\Gamma$ has no cycles, and hence by, \cite{hirsh}
Lemma 3.2, $\chi(\Gamma)\leq 2$. By Theorem \ref{chr. no.} (2),
$\mathfrak{Rd}_{df}\mathfrak{C}$ is not representable. It is easy to
show (e.g., because $\mathfrak{C}(\Gamma)$ is first-order
interpretable in $\Gamma$, for any $\Gamma$) that\begin{equation*}
\prod_{D}\mathfrak{C}(\Gamma_{\kappa})\cong\mathfrak{C}(\prod_{D}\Gamma_{\kappa}).\end{equation*}
Combining this with the fact that: for any $n$-dimensional atom
structure $\mathcal{S}$ \begin{center}$\mathcal{S}$ is strongly
representable $\Longleftrightarrow$ $\mathfrak{Cm}\mathcal{S}$ is
representable,\end{center}the desired follows.
\end{proof}
\bibliographystyle{plain}

\end{document}